\documentclass[11pt]{amsart}
\usepackage{amssymb,amsthm,amsmath}
\usepackage{graphics,psfrag,graphicx,subfigure,epsfig}
\usepackage{enumerate}
\usepackage[all]{xy}
\usepackage{mathtools} 
\usepackage{euscript}
\usepackage{eufrak}
\def\n{\mathbb{N}}

\def\f{\mathcal{F}}
\def\fo{\mathcal{F}}
\def\g{\mathcal{G}}
\def\fin{\mathcal{FIN}}
\def\vc{\mathcal{VC}}
\def\vci{\mathcal{VC}^\infty}

\def\evc{\underline{\underline{E}}G}

\def\efin{\underline{E}G}
\def\efine{\underline{E}}

\def\z{\mathbb{Z}}

\def\s{\mathbb{S}}

\DeclareMathOperator{\comm}{Comm}

\DeclareMathOperator{\vcd}{vcd}
\DeclareMathOperator{\ma}{Max}
\DeclareMathOperator{\Mod}{Mod}

\def\gd{\mathrm{gd}}
\def\cd{\mathrm{cd}}
\def\gdf{\underline{\mathrm{gd}}}
\def\gdvc{\underline{\underline{\mathrm{gd}}}}

\newtheorem{thm}{Theorem}[section]
\newtheorem{prop}[thm]{Proposition}

\newtheorem{theorem}{Theorem}
\newtheorem{proposition}[theorem]{Proposition}

\theoremstyle{definition}

\newtheorem{defi}[thm]{Definition}

\newtheorem{rem}[thm]{Remark}

\newtheorem{lem}[thm]{Lemma}
\newtheorem{cor}[thm]{Corollary}

\title[Classifying Spaces]{on the virtually-cyclic dimension  of surface braid groups and right-angled Artin groups}

\author{Alejandra Trujillo-Negrete}
\address{Centro de Investigaci\'on en Matem\'aticas, A. C.  Jalisco S/N, Col. Valenciana CP: 36023 Guanajuato, Gto, México}
\email{alejandra.trujillo@cimat.mx}

\begin{document}

\begin{abstract}

We give a bound for the virtually cyclic dimension of groups with a normal subgroup of finite index which satisfies that every infinite virtually-cyclic subgroup is contained in a unique maximal such subgroup. As an application we provide a bound for 
 the virtually-cyclic dimension for the braid group of a closed surface with genus greater than 2 and for right-angled Artin groups.  
\end{abstract}
\thanks{The author was supported by research grants from CONACyT-Mexico}

\maketitle
\section{Introduction} 

Let $G$ be a group. A \emph{family} $\fo$ of subgroups of $G$ is a set of subgroups of $G$ which is closed under conjugation and taking subgroups.  
 A model for the \textit{classifying space $E_{\fo}G$ for the family $\fo$} is a $G$-CW-complex $X$ such that the fixed point set $X^H$ is contractible for $H \in \fo$ and is empty if $H \not\in \fo$. 

Let $\fin_G$ and $\vc_G $ be the families of finite and virtually cyclic subgroups of $G$, respectively.  We abbreviate $\efin:=E_{\fin_G}G$ and $\evc:=E_{\vc_G}G$. The study of models for $\efin $ and $\evc$ has been motivated by the Baum-Connes and Farrel-Jones Conjectures. 

A model for $E_{\fo}G$ always exists and two models for $E_{\fo}G$ are $G$-homotopy equivalent (see \cite{luck}), however the  model may not have  finite dimension. The smallest possible dimension of a model for $E_{\f}G$ is called the \textit{geometric dimension of $G$
 for the family $\f$ } and is usually denoted as $\gd_{\f}G$.   
We abbreviate    $\gdf G:=\gd_{\fin_G}G$ and $\gdvc G: =\gd_{\vc_G}G$, they are called proper and virtually-cyclic dimension respectively, and for the trivial family $\{1\}$,  denote  $\gd G:=\gd_{\{1\}}G$.     

We remark that in general  $\gdf G \leq \gdvc G+1$   (see \cite{luck-weiermann}),   and for many families of subgroups the following inequality hold: 
\begin{equation}\label{eq-1}
\gdvc G  \leq \gdf G+1.
\end{equation}
In particular, L\"uck and Weiermann  showed in \cite{luck-weiermann} that if a group $G$ has pro-perty $\ma_{\vci_G}$  (which states that  every infinite virtually-cyclic subgroup is contained in a unique maximal such subgroup), then $G$ satisfies inequality (\refeq{eq-1}). In \cite{D-A} it was showed that the mapping class group of an orientable compact  surface has normal subgroups of finite index   $\Gamma$ with the property $\ma_{\vci_{\Gamma}}$.     \\

The principal result of this paper is the following: 
\begin{theorem}\label{main}
Let $1 \to H\to G \xrightarrow{\psi} F\to 1$ be a short exact sequence of groups where $F$ is a non-trivial finite group.  If $G$ is torsion free, $H$ has property $\ma_{\vci_H}$ and $\gdvc H<\infty$, then  $\gdvc G\leq \max\{3,\gdf H +1\}$.  
\end{theorem}
 Also we will give a bound for $\gdvc G$ when $G$ has torsion. We remark that Theorem \ref{main} extends some of the results that we obtained in \cite{D-J-A},  in which we showed that in many cases, inequality (\refeq{eq-1})  holds for the whole mapping class group, we used that the mapping class groups have the property of uniqueness of roots and an extra condition. 
 
 Below are some applications of  Theorem \ref{main}. Let $S$ be a surface, denote by $B_n(S)$ the $n$-braid group of $S$. 
\begin{proposition}
If $n\geq 1$ and $S$ is an orientable closed surface of genus $g\geq 2$, then 
   $ \gdvc B_n(S)\leq n+2$.   
\end{proposition}

\begin{proposition}
If $A$ is a right-angled Artin group, then $$\gdvc A\leq \gd A+1.$$ 
 In addition, for any subgroup $H\subseteq A$,  $\gdvc H\leq \gd H+1$. 
\end{proposition} 

\section{Classifying spaces for families}

A method for constructing  a model for $\evc$ is to start from a model for $\efin$ and then promote the action to a larger space to get a model for $\evc$.
In this section we will describe the model of L\"uck and Weiermann given in \cite{luck-weiermann}.
  
Define an equivalence relation on   the set $\vci_G=\vc_G -\fin_G$ of  infinite virtually cyclic subgroups of $G$. Given $V,U\in \vci_G$  
\begin{align}\label{rel-eq}
 V\sim U \;\text{if and only if } \; |V\cap U|=\infty .
 \end{align} 
Let $[\vci_G]$ denote the set of equivalence classes and by $[V]$ the equivalence class of  $V\in \vci_G$.     
For  $[V]\in \vci_G$ define
  \begin{align} \label{def-comm}
   \comm_{G}[V]:=\{g\in G \mid [gVg^{-1}]=[ V] \}, 
  \end{align}
   the commensurator of $V$ in $G$. 
 Define a family of subgroups of $\comm_G[V]$ as  
   \begin{align}\label{gv}
   \g_G[V]=\{ U\in \vci_{\comm_G[V]} \mid  [V]=[U]\} \cup \fin_{ \comm_G[V]}.
   \end{align}
   

\begin{thm}\cite[Thm.2.3]{luck-weiermann} \label{thm-luck-weierm}
Let $\vci_G$ and $\sim$ be as above. Let $I$ be a complete system of representatives $[V]$ of the $G$-orbits in $[\vci_G]$ under the $G$-action coming from conjugation. Choose arbitrary $\comm_G[V]$-CW-models for $\underline{E}(\comm_G[V])$,  $E_{\g[V]}(\comm_G[V])$ and an arbitrary $G$-CW-model for $\underline{E}(G)$. Define $X$ a $G$-CW-complex by the cellular $G$-pushout
\begin{align}\label{pushout-l}
\xymatrix{
                \coprod_{[V]\in I} G \times_{\comm_G[V]} \underline{E}(\comm_G[V])
                \ar[d]^{\coprod_{[V]\in I} id_G \times_{\comm_G[V]}f_{[V]}}
                \ar[r]^(0.75){i}
                &\efin
                \ar[d]
                \\
                \coprod_{[V]\in I} G \times_{\comm_G[V]} E_{\g[V]}(\comm_G[V])
                \ar[r]
                & X
}
\end{align}
such that $f_{[V]}$ is a cellular $\comm_G[V]$-map for every $[V]\in {I}$ and $i$ is an inclusion of $G$-CW-complexes, or such that every map $f_{[V]}$ is an inclusion of $\comm_G[V]$-CW-complexes for every $[V]\in I$  and $i$ is a cellular $G$-map. 
Then $X$ is a model for $\evc.$
\end{thm}

\begin{rem}\cite[Rem .2.5]{luck-weiermann} \label{rem-luck-weierm}
  Suppose that there exists  $n$ such that  $\gdf G\leq n$,  and   for each $[V]\in \mathfrak{I}$, 
$$ \gdf \comm_{G}[V] \leq n-1 \; \text{ and }\;
 \gd_{\g[V]}\comm_{G}[V] \leq n,
$$
 then $\gdvc G\leq n$. 
\end{rem}

\subsection{Proper  dimension}
The following Theorems will be used for the proof of the main Theorem, which is given in Section \ref{sec-vc}. 

Let $G$ be a group, and $F\in \fin_G$ a finite subgroup. The {\em length} $l(F)$ of $F$ is defined as the largest natural number $k$ for which there is a chain $1=F_0<F_1 < \cdots < F_k =F$.
The {\em length} of $G$ is 
$$l(G)= \sup \{l(F) \mid F\in \fin_G \}.$$
\begin{thm}\cite[Thm. 3.10, Lem. 3.9]{Conchita} \label{thm-gd-l}
 Suppose that $ \gdf G < \infty $. If $l(G)$ is finite, then    
\begin{align*}
\gdf G \leq \max\{3, \vcd G+l(G)\},
\end{align*}  
where $\vcd(G)$ denotes virtual cohomological dimension of $G$. 
\end{thm}

\begin{thm}\cite[Thm. 2.5]{Conchita-Eul}\label{thm-euler}
	Let $G$ be a group such that any finite  subgroup is nilpotent and   $\vcd G <\infty$.  Let  $n=\max_{F\in \fin_G} \{\vcd G + rk(W_G F) \}$, 
	where  $rk(\cdot)$ denotes the biggest rank of a finite  elementary abelian subgroup, then 
	$\; \gdf G \leq \max\{3,n\}. $
\end{thm}


\section{Virtually-cyclic dimension}\label{sec-vc}

We say that $G$ has \emph{property $\ma_{\vci_G}$} if
every  $V\in \vci_G$ is contained in a unique maximal $V_{max_G}\in \vci_G$.
 
 Suppose $G$ satisfies  $\ma_{\vci_G}$, let $V,U\in \vci_G$, thus $V\sim U$ if only if 
$V_{max_G} =U_{max_G}$, therefore  
$$\comm_G[V]=N_G(V_{max_G})$$
is the normalizer of $V_{max_G}$ in $G$ (see \cite{luck-weiermann}). 
\begin{lem}\label{lem-0}
If $G$ has property $\ma_{\vci_G}$ and  $U\in \vci_G$, then  \begin{align*}(gUg^{-1})_{max_G} =gU_{max_G}g^{-1}.
\end{align*}
 \end{lem}
\begin{proof}
We first claim that $gU_{max_G}g^{-1}$ is maximal in $\vci_G$: let $V\in \vci_G$  such that    $gU_{max_G}g^{-1} \subseteq V$, hence   $U_{max_G}\subseteq g^{-1}Vg$ and by maximality we have that  $U_{max_G}=g^{-1}Vg$.  In this way,  $gU_{max_G}g^{-1} =V$.   So we conclude that $(gUg^{-1})_{max_G} =gU_{max_G}g^{-1}$ by the uniqueness requirement of the property $\ma_{\vci_G}$.

\end{proof}

\begin{lem} \label{lem-1}
Let $1 \to H\to G \xrightarrow{\psi} F\to 1$ be a short exact sequence of groups, such that $H$ satisfies property $\ma_{\vci_H}$ and $F$ is finite. Let $V\in \vci_G$  and $(V\cap H)_{max_H}\in \vci_H$ be the maximal containing $(V\cap H)$,  then 
  $$\comm_G[V]=N_G((V\cap H)_{max_H}).$$
 \end{lem}
\begin{proof} Let $g\in G$,  we have that
\begin{align}
[gVg^{-1}]=[ V] &\text{ if and only if } \; [g(V\cap H)g^{-1}]=[V\cap H] \nonumber
\\
&\text{ if and only if }\; (g(V\cap H)g^{-1})_{max_H}=(V\cap H)_{max_H}, \label{eq-vcaph}
\end{align}
by Lemma \ref{lem-0} $(g(V\cap H)g^{-1})_{max_H}=g(V\cap H)_{max_H} g^{-1}$, then 
equality (\ref{eq-vcaph}) holds  if and only if  $g(V\cap H)_{max_H} g^{-1}=(V\cap H)_{max_H}$.

\end{proof}

\begin{thm}\label{thm-ex-max}
Let $1 \to H\to G \xrightarrow{\psi} F\to 1$ be a short exact sequence of groups where $F$ is a non-trivial finite group.  Suppose that $H$ has property $\ma_{\vci_H}$ and $\gdvc H<\infty$, then 
 $\gdvc G\leq \max\{3,\gdf H +l(F)\}$.  
 If in addition $G$ is torsion free, then $\gdvc G\leq \max\{3,\gdf H +1\}$.  
\end{thm}
\begin{proof} 
Let $V\in \vci_G$  and  $U=(V\cap H)_{max_H}$, by Lemma \ref{lem-1} we have that $\comm_G[V]=N_G(U)$. Thus a model for 
 $\efine N_G(U)/U $ is a model for $E_{\g_G[V]}N_G(U)$ with the $N_G(U)$-action induced by the projection $p\colon N_G(U) \to N_G(U)/U$. Hence $\gd_{\g_G[V]}N_G(U) \leq \gdf N_G(U)/U$.  

We will use Theorem \ref{thm-gd-l} to give a bound for $\gdf N_G(U)/U$.   From the exact sequence
\begin{align*}
\xymatrix{
1\ar[r] & N_H(U) \ar[r] & N_G(U) \ar[r]^(0.7){\psi|}& 
F' \ar[r] & 1, }
\end{align*}
where $\psi|$ is the restriction of $\psi$  and  $F'\subseteq F$, we have  
\begin{align*}
\xymatrix{
1\ar[r] & N_H(U)/U \ar[r] & N_G(U)/U \ar[r]^(0.7){\overline{\psi}}& 
F' \ar[r] & 1, }
\end{align*}
so $\vcd(N_G(U)/U)= \vcd(N_H(U)/U)\leq \gdf (N_H(U)/U)$. Since $H$ satisfies property $\ma_{\vci_H}$ and $\gdvc H<\infty$, by   the proof of \cite[Thm. 5.8]{luck-weiermann} we have that  $\gdf(N_H(U)/U) \leq \gdf N_H(U)$,  hence    
$$\vcd (N_G(U)/U)\leq \gdf N_H(U)\leq \gdf H,$$ 
we remark that $\gdf H \leq \gdvc H+1 <\infty$. 
\\ 
Further, $N_H(U)/U$ is torsion free because $U$ is virtually cyclic maximal in $N_H(U)$, therefore the finite subgroups of $N_G(U)/U$ embeds in $F'$, so that $$l(N_G(U)/U)\leq l(F')\leq l(F)<\infty.$$ Applying Theorem \ref{thm-gd-l} we have that  
$$\gdf N_G(U)/U \leq \max\{3, \gdf H +l(F)\}.$$
By Theorem \ref{thm-luck-weierm} and Remark \ref{rem-luck-weierm}, we may conclude that $$\gdvc G\leq \max\{3,\gdf H +l(F)\}.$$
\\
If $G$ is torsion free, then the finite subgroups of $N_G(U)/U$ are cyclic. Applying Theorem  \ref{thm-euler} we have 
\begin{align*}
\gdf N_G(U)/U &\leq \max\{3, \vcd N_G(U)/U +1\} 
\\     &\leq \max\{3, \gdf H +1\},
\end{align*}
 again by Theorem \ref{thm-luck-weierm} and  Remark  \ref{rem-luck-weierm}, we obtain the following inequality   $$\gdvc G\leq \max\{3, \gdf H +1 \}.$$ 
\end{proof}

\section{Applications } 

We will use the following Lemma and the fact that surface braid groups and  right-Artin groups can be embeded in some mapping class group of a surface, to show that they have normal subgroups  of finite index with property  of maximality in the family of virtually cyclic subgroups, and finally apply  Theorem \ref{thm-ex-max}. 

\begin{lem}\label{subgr-max}
Let $G$ be a group that satisfies property $\ma_{\vci_G}$, if $H\subseteq G$, then $H$ has property $\ma_{\vci_H}$. 
\end{lem} 
The proof of the Lemma is left to the reader. \\
 
\subsection{Mapping class groups}
Let   $S$ be  a connected, compact,   orientable surface  with a finite set of   points removed from the interior (punctures).
Denote by $\Mod (S)$ the \textit{mapping class group}  of $S$.

Let $m>1$ be an integer number,   there is a natural homomorphism
$$\tau\colon \Mod(S) \to Aut(H_1(S;\z_m))$$
defined by the action of diffeomorphisms on the homology group.  The subgroup 
$$\Mod(S)[m]=\ker {\tau},$$ is called the $m$-\textit{congruence subgroup of} $\Mod(S)$. 
Note that this subgroup  has
finite index in $\Mod(S)$. 

\begin{thm}\cite[Prop. 5.11, Thm. 5.3]{D-A} \label{thm-max-mod}
Let $S$ be an orientable compact surface with finitely many punctures and $\chi(S)\leq 0$, then $\gdvc Mod(S)$ is finite. Furthermore, for $m\geq 3$ the group $\Mod(S)[m]$ satisfies property  $\ma_{\vci_{\Mod(S)[m]}}$. 
\end{thm}

\begin{cor}\label{subtf-mod}
Let $S$ be an orientable compact surface with finitely many punctures and $\chi(S)\leq 0$. If $H$ is a torsion free subgroup of $Mod(S)$, then $\gdvc H\leq \gdf H +1$. 
\end{cor}
\begin{proof}
Let $H$ be a torsion free subgroup of $Mod(S)$ and   $H_m=H\cap \Mod(S)[m]$, with $m\geq 3$. By Theorem \ref{thm-max-mod} and Lemma \ref{subgr-max}, $H_m$ has property $\ma_{\vci_{H_m}}$.  Since $H_m$ is a normal finite index subgroup of $H$, by applying  Theorem \ref{thm-ex-max}, we conclude that $\gdvc H\leq \gd H +1$. 
\end{proof}

\subsection{Surface braid groups}

Let $S$ be a compact orientable surface. The \textit{$n$-configuration space of $S$} is defined as follows,
 $$F_n(S)= \{(y_1,...,y_n)\in S^n \mid y_i\neq y_j \text{ for all }i,j\in\{1,...,n\}, i\neq j\}.$$
 We endow $F_n(S)$ with the topology induced by the product topology from  the space $S^n$. The configuration space $F_n(S)$ is a connected $2n$-dimensional open manifold. There is a natural free action of the symmetric group $\Sigma_n$ on $F_n(S)$ by permuting the coordinates. We will denote the quotient  by the action as  $D_n(S)=F_n(S)/\Sigma_n$ and it may be thought of as the configuration space of $n$ unordered points. 
\begin{defi}
Let $n\in \n$. The \emph{$n$-braid group} of $S$ is defined as  $$B_n(S)=\pi_1(D_n(S)).$$ 
\end{defi}
If $S$ is a closed orientable surface of genus $>1$, $D_n(S)$ is a model for $E B_n(S)$, \cite{Fadell-Neuwirth}.  We remark that  braid groups  $B_n(S)$ are torsion free if and only if $S$ is different from $\s^2$ (see \cite{VanBuskirk}), in that case $\gd B_n(S)$ coincides with the cohomological dimension $\cd B_n(S)$, except possibly for the case where $\cd(B_n(S)) = 2$ and $\gd(B_n(S)) = 3$. 
\begin{thm}\cite[Thm. 1.2]{GGM} \label{thm-gdfbn}
If $n\geq 1$ and    $S$ is  a  closed   orientable surface of genus $g\geq 1$,  then $\cd B_n(S)=n+1$. 
\end{thm}

The mapping class groups are closely related to braid groups, see \cite{Birman}.  
Let   $S$ be a closed orientable surface of genus $\geq 2$,   and $Q_n=\{x_1,...,x_n\}$ be a set with $n\geq 1$ different points  in the interior of  $S$, then \begin{align}\label{suc-bn-mod}
\xymatrix{
1\ar[r] & B_n(S)  \ar[r] &  \Mod(S-Q_n)  \ar[r]^(0.65){\rho}& \Mod(S) \ar[r] & 1 . }
\end{align}    

\begin{prop}
If $n\geq 1$ and $S$ is an orientable closed surface of genus $g\geq 2$, then 
   $ \gdvc B_n(S)\leq n+2$.   
\end{prop}
\begin{proof}
If $n=1$, then $B_1(S)=\pi_1(S)$  is hyperbolic and $\gdvc B_1(S)=2$.  \\
If $n\geq 2$, then $\cd B_n(S)\geq 3$ and $\cd B_n(S)=\gd B_n(S)$.  We consider $B_n(S)$ as  a subgroup of $Mod(S-Q_n)$  via the inclusion of the exact sequence (\ref{suc-bn-mod}). 
  By  Corollary \ref{subtf-mod} and Theorem \ref{thm-gdfbn} we have   
$$\gdvc B_n(S)\leq \gd B_n(S)+1=n+2. $$ 
\end{proof}

\subsection{Right angled Artin groups}

Artin groups arose as a natural generalization of braid groups $B_n=B_n(\mathbb{D}^2)$.\emph{ A right-angled Artin group} $A$ is a group with presentation of the form 
$$A= \langle s_1,...,s_n \mid \underbrace{s_is_js_i...}_{m_{ij}}=\underbrace{s_js_is_j...}_{m_{ij}} \;\; \text{for all }\; i\neq j\rangle,  $$
where $m_{ij}=m_{ji} \in \{2,\infty\}$ for all $i,j$, when $m_{ij}=\infty$   we omit the relation  between $s_i$ and $s_j$.  



We remark that right-angled Artin groups contain many interesting groups: some $3$-manifold groups, surface groups \cite{C-W} and graph braid groups. 

It is well known that for any right-angled Artin group $A$, the geometric dimension $\gd A$ is finite,  and $A$ is torsion free.  
 
\begin{prop}\label{raag-mcg}\cite{Koberda}
Every right-angled Artin group embeds in the mapping class group of any surface of sufficiently high genus.  
\end{prop}

\begin{prop} 
If $A$ is a right-angled Artin group, then $$\gdvc A\leq \gd A+1.$$ 
 In addition, for any subgroup $H\subseteq A$,  $\gdvc H\leq \gd H+1$. 
\end{prop} 
\begin{proof}
The proof follows directly by  Corollary \ref{subtf-mod} and Proposition \ref{raag-mcg}.
\end{proof} 






\vspace{3cm}
\begin{thebibliography}{}


\bibitem{D-J-A}{J. Aramayona, D. Juan-Pineda, A. Trujillo-Negrete. On the virtually-cyclic dimension of mapping class groups of punctured spheres,  	arXiv:1708.03709, to appear in Algebraic and Geometric Topology. 
}
\bibitem{Birman}{J.  S. Birman. Braids, links and mapping class groups, Ann. Math. Stud. 82, Princeton University Press, 1974.}

\bibitem{C-W}{J. Crisp and B. Wiest.  Embeddings of graph braid and surface groups in right-angled Artin groups and
braid groups, Algebr. Geom. Topol. 4 (2004), 439–472.}

\bibitem{D-J}{M. Davis and T. Januszkiewicz. Right-angled Artin groups are commensurable with right-angled Coxeter groups, J. Pure Appl. Algebra 153 (2000), no. 3, 229–235.}



\bibitem{Fadell-Neuwirth} {E. Fadell, L. Neuwirth. Configuration spaces, Math. Scand. 10 (1962) 111–118.}

\bibitem{D-A}{D. Juan-Pineda,  A. Trujillo-Negrete. On Classifying Spaces for the Family of Virtually Cyclic Subgroups in Mapping Class Groups, Pure and Applied Mathematics Quarterly Volume 12, Number 2, 261–292, 2016}

\bibitem{Koberda}{Koberda, T. Right-angled Artin groups and a generalized isomorphism problem for finitely generated subgroups of mapping class groups,  Geom. Funct. Anal. (2012) 22: 1541}

\bibitem{GGM}{D. Lima Gon\c calves, J. Guaschi, M. Maldonado. Embeddings and the (virtual) cohomo- logical dimension of the braid and mapping class groups of surfaces. 2016. $<$hal-01377681$>$}

\bibitem{luck}{W. L\"uck. { Survey on classifying spaces for families of subgroups.}  In infinite groups: geometric, combinatorial and dynamical aspects, volume 248 of Progr. Math., pages 269-322. Birkh\"auser, Basel, 2005.} 

\bibitem{luck-weiermann}{W. L\"uck., and M. Weiermann. {\em On the classifying space of the family of virtually cyclic subgroups.}
Pure and Applied Mathematics Quarterly, \textbf{8} Nr. 2 (2012) 497--555.} 
￼￼
\bibitem{Conchita-Eul}{C. Mart\'inez-P\'erez. Euler classes and Bredon cohomology for groups with restricted families of finite subgroups, Math. Z. 275 (2013), 761—780.}

\bibitem{Conchita} C. Mart\'inez-P\'erez, {\em A bound for the Bredon cohomological dimension.}
J. Group Theory, \textbf{10}  (2007) 731-747.

\bibitem{VanBuskirk}{J. Van Buskirk. Braid groups of compact 2-manifolds with elements of finite order, Trans.  Amer. 
Math. Soc. 122 (1966), 81–97.}





 
 \end{thebibliography}
\end{document}